\documentclass[12pt,twoside]{article}

\newtheorem{thm}{Theorem}[section]
\newtheorem{proof}{Proof}
\newtheorem{defi}{Definition}
\newtheorem{cor}{Corollary}

\date{}

\begin{document}


\centerline{}

\centerline{}

\centerline {\Large{\bf Some Properties of Kenmotsu Manifolds Admitting  }}


\centerline{\Large{\bf a Semi-symmetric Non-metric Connection}}

\centerline{}

\centerline{\bf {S. K. Chaubey}}

\centerline{Section of Mathematics, Department of IT, Shinas college of technology,}

\centerline{Shinas,  P.O. Box 77, Postal Code 324, Sultanate of Oman.}

\centerline{Email: sk$22_{-}$math@yahoo.co.in}


\centerline{}

\centerline{\bf {A. C. Pandey}}


\centerline{Department of Mathematics, Bramanand P. G. College, Kanpur$-208004$, U. P., India. }

\centerline{Email: acpbnd$73$@gmail.com}
\centerline{}

\centerline{\bf {N. V. C. Shukla}}


\centerline{Department of Mathematics and Astronomy, Lucknow University -226007, U.P., India. }

\centerline{Email: nvcshukla$72$@gmail.com}
\begin{abstract}
The aim of this paper is to study generalized recurrent, generalized Ricci-recurrent, weakly symmetric and weakly Ricci-symmetric Kenmotsu manifolds with respect to the semi-symmetric non-metric connection.
\end{abstract}


\section{Introduction}
Let $(M_{n},g)$ be a Riemannian manifold of dimension $n$. A linear connection $\nabla$ in $(M_{n},g)$, whose torsion tensor $T$ of type $(1,2)$ is defined as
\begin{equation}
T(X,Y)=\nabla_{X}Y-\nabla_{Y}X-[X,Y],
\end{equation}
for arbitrary vector fields $X$ and $Y$, is said to be torsion free or symmetric if $T$ vanishes, otherwise it is non-symmetric. If the connection $\nabla$ satisfy $\nabla{g}=0$ in $(M_{n},g)$, then it is called metric connection otherwise it is non-metric. Friedmann and Schouten \cite{friedmann} introduced the notion of semi-symmetric linear connection on a differentiable manifold. Hayden \cite{hayden} introduced the idea of semi-symmetric linear connection with non-zero torsion tensor on a Riemannian manifold. The idea of semi-symmetric metric connection on Riemannian manifold was introduced by Yano \cite{yano}. He proved that a Riemannian manifold with respect to the semi-symmetric metric connection has vanishing curvature tensor if and only if it is conformally flat. This result was generalize for vanishing Ricci tensor of the semi-symmetric metric connection by T. Imai (\cite{imai1}, \cite{imai2}). Various properties of such connection have studied in (\cite{add4}, \cite{add5}) and by many other geometers. Agashe and Chafle \cite{agashe1} defined and studied a semi-symmetric non-metric connection in a Riemannian manifold. This was further developed by Agashe and Chafle \cite{agashe2}, De and Kamilya \cite{de1}, Pandey and Ojha \cite{pandey}, Chaturvedi and Pandey \cite{chaturvedi} and others. Sengupta, De and Binh \cite{sengupta}, De and Sengupta \cite{de2} defined new type of semi-symmetric non-metric connections on a Riemannian manifold and studied some geometrical properties with respect to such connections. In this connection, the properties of  non-metric connections have studied in (\cite{ozgur3}, \cite{kumar}, \cite{dubey}, \cite{add1}, \cite{add2}) and many others. In $2008$, Tripathi introduced the generalized form of a new connection in Riemannian manifold \cite{tripathi2}. Chaubey \cite{chaubey1, chaubeyadd1} defined semi-symmetric non-metric connections on an almost contact metric manifold and studied its different geometrical properties. Some properties of such connections have been noticed in (\cite{jaiswal}, \cite{chaubey3}, \cite{skchaubey_2011}, \cite{addsk4}, \cite{add6}) and others.
\newline 
In 1972, K. Kenmotsu \cite{kenmotsu} introduced a class of contact Riemann manifold known as Kenmotsu Manifold. He studied that if a Kenmotsu manifold satisfies the condition R(X,Y)Z = 0, then the manifold is of negative curvature -1, where R is the Riemannian curvature tensor of type (1,3) and R(X,Y)Z is derivative of tensor algebra at each point of the tangent space. Several properties of Kenmotsu Manifold have been studied by  Sinha and Srivastav \cite{srivastava}, De \cite{ucd}, De and Pathak \cite{pathak}, Chaubey et al. (\cite{chaubeyadd2}, \cite{addsk3}, \cite{skchaubey_2012}) and many others. Ozgur \cite{ozgur2} studied generalised recurrent Kenmotsu manifold and proved that if M be a generalised recurrent Kenmotsu manifold and generalised Ricci Recurrent Kenmotsu manifold then $\beta = \alpha$ holds on M. Sular studies the generalised recurrent and generalised Ricci recurrent Kenmotsu manifolds with respect to semi symmetric metric connection and proved that $\beta = 2\alpha$ where $\alpha$ and $\beta$ are smooth functions and M is generalised recurrent and generalised Ricci recurrent Kenmotsu manifold admitting a semi-symmetric connection \cite{sular}. In the present paper, we studied the properties of semi-symmetric non-metric connection in Kenmotsu manifolds.
\newline
The present paper is organized as follows. Section $2$ is preliminaries in which basic concepts of Kenmotsu manifolds are given. Section $3$ deals with the brief account of semi-symmetric non-metric connection. In section $4$, we define a generalized recurrent Kenmotsu manifolds with respect to the semi-symmetric non-metric connection and studied its some properties. Section $5$ is concerned with the weakly symmetric Kenmotsu manifolds with respect to the semi-symmetric non-metric connection.
\section{Preliminaries}
An $n$-dimensional Riemannian manifold $(M_n,g)$ of class $C^{\infty}$ with a $1$-form $\eta$, the associated vector field $\xi$ and a $(1,1)$ tensor field $\phi$ satisfying
\begin{equation}
\label{2.1}
{\phi}^2{X}+X={\eta(X)}{\xi},
\end{equation}
\begin{equation}
\label{2.2}
{\phi}{\xi}=0,\hspace{.5cm}{\eta({\phi{X}})}=0,\hspace{.5cm}{\eta(\xi)}=1,
\end{equation}
for arbitrary vector field $X$, is called an almost contact manifold. This system $({\phi},{\xi},{\eta})$ is called an almost contact structure to $M_n$  \cite{blair}. If the associated Riemannian metric $g$ in $M_n$ satisfy
\begin{equation}
\label{2.3}
g({\phi}X,{\phi}Y)=g(X,Y)-{\eta}(X){\eta}(Y),
\end{equation}
for arbitrary vector fields $X$, $Y$ in $M_n$, then $(M_n,g)$ is said to be an almost contact metric manifold. Putting ${\xi}$ for $X$
in (\ref{2.3}) and using (\ref{2.2}), we obtain
\begin{equation}
\label{2.4}
g({\xi},Y)={\eta}(Y).
\end{equation}
Also,
\begin{equation}
\label{2.5}
\varphi(X,Y){\stackrel{\mathrm{def}}{=}}g({\phi}X,Y)
\end{equation}
gives
\begin{equation}
\label{2.6}
\varphi(X,Y)+\varphi(Y,X)=0.
\end{equation}
where $\varphi = d\eta$ is 2-form.
\newline
	If moreover
\begin{equation}
\label{2.7}
(D_X{\phi})(Y)=g({\phi{X}},Y){\xi}-{\eta({Y})}{\phi{X}} ,
\end{equation}
\begin{equation}
\label{2.8}
D_X{\xi}=X-{\eta({X})}{\xi} ,
\end{equation}
hold in $(M_n,g)$, where $D$ being the Levi-Civita connection of the Riemannian metric $g$, then $(M_n,g)$ is called a Kenmotsu manifold \cite{kenmotsu}.
Also the following relations hold in a Kenmotsu manifold
\begin{equation}
\label{2.9}
K(X,Y){\xi}=\eta({X})Y-\eta{(Y)}X,
\end{equation}
\begin{equation}
\label{2.10}
K({\xi},X)Y=\eta(Y){X}-g(X,Y){\xi},
\end{equation}
\begin{equation}
\label{2.11}
S(X,{\xi})=-(n-1){\eta(X)},
\end{equation}
\begin{equation}
\label{2.12}
(D_X{\eta})(Y)=g(X,Y)-{\eta(X)}{\eta(Y)}
\end{equation}
for arbitrary vector fields $X$ and $Y$, where $K$ and $S$ denote the Riemannian curvature and Ricci tensors of the connection $D$ respectively.

\section{Semi-symmetric non-metric connection}

A linear connection $\nabla$ on $(M_n,g)$ is said to be a semi-symmetric non-metric connection if the torsion tensor $T$ of the connection $\nabla$ and the Riemannian metric $g$ satisfy the following conditions:
\begin{equation}
\label{3.1}
T(X,Y)=2\varphi(X,Y)\xi,
\end{equation}
\begin{equation}
\label{3.2}
(\nabla_{X}g)(Y,Z)=-\eta(Y)\varphi(X,Z)-\eta(Z)\varphi(X,Y),
\end{equation}
for arbitrary vector fields $X$, $Y$ and $Z$, where $\eta$ is $1-$form on $(M_{n},g)$ with $\xi$ as associated vector field.
If $D$ denotes the Levi-Civita connection, then the semi-symmetric non-metric connection \cite{chaubeyadd1,chaubey3} $\nabla$ on $(M_{n},g)$ is defined as
\begin{equation}
\label{3.3}
\nabla_{X}Y=D_{X}Y+g(\phi{X},Y)\xi,
\end{equation}
for arbitrary vector fields $X$ and $Y$.

	The curvature tensor $R$ of the semi-symmetric non-metric connection \cite{chaubey3} $\nabla$ is defined as
\begin{eqnarray}
\label{3.4}
R(X,Y)Z&=&K(X,Y)Z+g(\phi{Y},Z)D_{X}\xi-g(\phi{X},Z)D_{Y}\xi
\nonumber\\&&+g\left( (D_{X}\phi)(Y)-(D_{Y}\phi)(X),Z\right)\xi.
\end{eqnarray}
From (\ref{2.2}), (\ref{2.4}), (\ref{2.7}) and (\ref{2.8}), it follows that
\begin{equation}
\label{3.5}
R(X,Y)Z=K(X,Y)Z+g(\phi{Y},Z)X-g(\phi{X},Z)Y+2\eta(Z)g(\phi{X},Y)\xi
\end{equation}
which give
\begin{equation}
\label{3.6}
\tilde{S}(Y,Z)=S(Y,Z)+(n-1)g(\phi{Y},Z)
\end{equation}
and
\begin{equation}
\label{3.7}
\tilde{r}=r.
\end{equation}
Here $\tilde{S}$ and $\tilde{r}$  denote the Ricci tensor and scalar curvature with respect to the semi-symmetric non-metric connection $\nabla$ and $r$ is the scalar curvature with respect to the Levi-Civita connection $D$. From (\ref{3.7}) we leads the following corollary:
\begin{cor}
Let $M_{n}$ be an $n-$dimensional Kenmotsu manifold equipped with a semi-symmetric non-metric connection $\nabla$, then the scalar curvature with respect to semi-symmetric non-metric connection is equal to scalar curvature with respect to Levi-Civita connection.
\end{cor}
Replacing $Z$ by $\xi$ in (\ref{3.5}) and (\ref{3.6}) and then using (\ref{2.2}) and (\ref{2.4}), we get
\begin{equation}
\label{3.8}
R(X,Y)\xi=K(X,Y)\xi+2g(\phi{X},Y)\xi
\end{equation}
and
\begin{equation}
\label{3.9}
\tilde{S}(Y,\xi)=S(Y,\xi).
\end{equation}

\section{ Generalized Recurrent Kenmotsu Manifolds }
{\defi A non-flat $n-$dimensional differentiable manifold $M_{n}$, $(n>3)$, is called generalized recurrent manifold \cite{de3} if its curvature tensor $K$ satisfies the condition
\begin{equation}
\label{4.1}
(D_{X}K)(Y,Z)W=A(X)K(Y,Z)W+B(X)\left[ g(Z,W)Y-g(Y,W)Z\right] ,
\end{equation}
where $A$ and $B\neq{0}$ are $1-$forms defined as
\begin{equation}
\label{4.2}
A(X)=g(X,\rho_{1}),{\hspace{10.5pt}} B(X)=g(X,\rho_{2}),
\end{equation}
for arbitrary vector fields $X$, $Y$, $Z$ and $W$. Here $\rho_{1}$ and $\rho_{2}$ are the vector fields associated with the $1-$forms $A$ and $B$ respectively.}

{\defi A non-flat $n-$dimensional differentiable manifold $M_{n}$, $(n>3)$, is called generalized Ricci-recurrent \cite{de3} if its Ricci tensor $S$ satisfies the condition
\begin{equation}
\label{4.3}
(D_{X}S)(Y,Z)=A(X)S(Y,Z)+(n-1)B(X)g(Y,Z),
\end{equation}
for arbitrary vector fields $X$, $Y$ and $Z$, where $A$ and $B$ are defined as in (\ref{4.2}).}

	In the similar fashion, we defined the following definitions :
{\defi A non-flat $n-$dimensional differentiable manifold $M_{n}$, $(n>3)$, is called generalized recurrent with respect to the semi-symmetric non-metric connection $\nabla$ if its curvature tensor $R$ satisfies the condition
\begin{equation}
\label{4.4}
(\nabla_{X}R)(Y,Z)W=A(X)R(Y,Z)W+B(X)\left[ g(Z,W)Y-g(Y,W)Z\right] ,
\end{equation}
for arbitrary vector fields $X$, $Y$, $Z$ and $W$.

{\defi A non-flat $n-$dimensional differentiable manifold $M_{n}$, $(n>3)$, is called generalized Ricci-recurrent with respect to the semi-symmetric non-metric connection $\nabla$ if its Ricci tensor $\tilde{S}$ satisfies the condition
\begin{equation}
\label{4.5}
(\nabla_{X}\tilde{S})(Y,Z)=A(X)\tilde{S}(Y,Z)+(n-1)B(X)g(Y,Z),
\end{equation}
for arbitrary vector fields $X$, $Y$, $Z$, where $A$ and $B$ are defined as in (\ref{4.2}).}
Now we consider the generalized recurrent and generalized Ricci-recurrent Kenmotsu manifolds admitting the semi-symmetric non-metric connection $\nabla$ and prove the following theorems:
\begin{thm}
Let $M_{n}$ be an $n-$dimensional generalized recurrent Kenmotsu manifold equipped with a semi-symmetric non-metric connection $\nabla$. Then $B=A$ holds on $M_{n}$.
\end{thm}
\begin{proof}
Replacing $Y$ and $W$ by $\xi$ in (\ref{4.4}) and using (\ref{2.2}) and (\ref{2.4}), we obtain
\begin{equation}
\label{4.6}
(\nabla_{X}R)(\xi,Z)\xi=A(X)R(\xi,Z)\xi+B(X)\left[ \eta(Z)\xi-Z\right].
\end{equation}
In consequence of (\ref{2.9}) and (\ref{3.8}), (\ref{4.6}) becomes
\begin{equation}
\label{4.7}
(\nabla_{X}R)(\xi,Z)\xi=\left( B(X)-A(X)\right) \left[ \eta(Z)\xi-Z\right].
\end{equation}
It can be easily seen that
\begin{equation}
\label{4.8}
(\nabla_{X}R)(\xi,Z)\xi=\nabla_{X}R(\xi,Z)\xi-R(\nabla_{X}\xi,Z)\xi-R(\xi,\nabla_{X}Z)\xi-R(\xi,Z)\nabla_{X}\xi.
\end{equation}
From (\ref{2.2}), (\ref{2.9}), (\ref{3.8}) and (\ref{4.8}), it follows that
\begin{equation}
\label{4.9}
(\nabla_{X}R)(\xi,Z)\xi=0.
\end{equation}
In view of (\ref{4.7}) and (\ref{4.9}), we get
\begin{equation}
\label{4.10}
\left( B(X)-A(X)\right) \left[ \eta(Z)\xi-Z\right]=0.
\end{equation}
Since $Z\neq{\eta(Z)\xi}$ in general, therefore $B=A$. 
\end{proof}
\begin{thm}
If an $n-$dimensional generalized Ricci-recurrent Kenmotsu manifold $M_{n}$ admitting a semi-symmetric non-metric connection $\nabla$, then $B=A$ holds on $M_{n}$.
\end{thm}
\begin{proof}
Replacing $Z$ by $\xi$ in (\ref{4.5}) and then using (\ref{2.4}), (\ref{2.11}) and (\ref{3.9}), we obtain
\begin{equation}
\label{4.11}
(\nabla_{X}\tilde{S})(Y,\xi)=(n-1)\eta(Y)\left[ B(X)-A(X)\right].
\end{equation}
It is obvious that
\begin{equation}
\label{4.12}
(\nabla_{X}\tilde{S})(Y,\xi)=\nabla_{X}\tilde{S}(Y,\xi)-\tilde{S}(\nabla_{X}Y,\xi)-\tilde{S}(Y,\nabla_{X}\xi).
\end{equation}
In consequence of (\ref{2.2}), (\ref{2.4}), (\ref{2.11}), (\ref{2.12}), (\ref{3.3}) and (\ref{3.9}), (\ref{4.12}) becomes
\begin{equation}
\label{4.13}
(\nabla_{X}\tilde{S})(Y,\xi)=-(n-1)g(X,Y)+2(n-1)g(\phi{X},Y)-S(X,Y).
\end{equation}
From (\ref{4.11}) and (\ref{4.13}), it follows that
\begin{equation}
\label{4.14}
(n-1)\eta(Y)\left[ B(X)-A(X)\right]=-(n-1)g(X,Y)+2(n-1)g(\phi{X},Y)-S(X,Y).
\end{equation}
Putting $Y=\xi$ in (\ref{4.14}) and using (\ref{2.2}), (\ref{2.4}) and (\ref{2.11}), we obtain $B=A$. 
\end{proof}

\section{ Weakly symmetric Kenmotsu manifolds }
{\defi A non-flat $n-$dimensional differentiable manifold $M_{n}$, $(n>3)$, is called pseudo symmetric \cite{chaki} if there is a $1-$form $A$ on $M_{n}$ such that
\begin{eqnarray}
\label{5.1}
(D_{X}K)(Y,Z)W&=&2A(X)K(Y,Z)W+A(Y)K(X,Z)W+A(Z)K(Y,X)W\nonumber\\&&+A(W)K(Y,Z)X+g(K(Y,Z)W,X)\rho_{1},
\end{eqnarray}
where $D$ is the Levi-Civita connection and $X$, $Y$, $Z$ and $W$ are arbitrary vector fields on $M_{n}$. The vector field $\rho_{1}$ associated with the $1-$form $A$ is defined by $A(X)=g(X,\rho_{1})$.}

	{\defi A non-flat $n-$dimensional differentiable manifold $M_{n}$, $(n>3)$, is called weakly symmetric \cite{tamassy1, tamassy2} if there are $1-$forms $A$, $B$, $C$ and $D$ on $M_{n}$ such that
\begin{eqnarray}
\label{5.2}
(D_{X}K)(Y,Z)W&=&A(X)K(Y,Z)W+B(Y)K(X,Z)W+C(Z)K(Y,X)W\nonumber\\&&+D(W)K(Y,Z)X+g(K(Y,Z)W,X)\sigma,
\end{eqnarray}
where $X$, $Y$, $Z$, $W$ are arbitrary vector fields on $M_{n}$. The vector field $\sigma$ associated with the $1-$form $p$ is defined as $p(X)=g(X,\sigma)$. A weakly symmetric manifold $M_{n}$ is said to be pseudo symmetric if $B=C=D=A$, $\sigma=\rho_{1}$ and $A$ is replaced by $2A$, locally symmetric if $A=B=C=D=0$ and $\sigma=0$. A weakly symmetric manifold is said to be proper if at least one of the $1-$forms $A$, $B$, $C$ and $D$ is non zero or $\sigma{\neq}0$.}

	{\defi A non-flat $n-$dimensional differentiable manifold $M_{n}$, $(n>3)$, is called weakly Ricci-symmetric \cite{tamassy1, tamassy2} if there are $1-$forms $\alpha$, $\beta$ and $\gamma$ on $M_{n}$ such that
\begin{equation}
\label{5.3}
(D_{X}S)(Y,Z)=\alpha(X)S(Y,Z)+\beta(Y)S(X,Z)+\gamma(Z)S(X,Y),
\end{equation}
where $X$, $Y$ and $Z$ are arbitrary vector fields on $M_{n}$. A weakly Ricci-symmetric manifold $M_{n}$ is called pseudo Ricci-symmetric if $\alpha=\beta=\gamma$.}
Contracting (\ref{5.2}) with respect to $Y$, we get
\begin{eqnarray}
\label{5.4}
(D_{X}S)(Z,W)&=&A(X)S(Z,W)+B(K(X,Z)W)+C(Z)S(W,X)\nonumber\\&&+D(W)S(X,Z)+p(K(X,W)Z),
\end{eqnarray}
where $p$ is defined as $p(X)=g(X,\sigma)$ for arbitrary vector field $X$. The author \cite{add3} studied the properties of weakly and weakly Ricci symmetric manifolds with examples.

Similarly we define the following definitions:

	{\defi A non-flat $n-$dimensional differentiable manifold $M_{n}$, $(n>3)$, is called weakly symmetric with respect to the semi-symmetric non-metric connection $\nabla$ if there are $1-$forms $A$, $B$, $C$ and $D$ on $M_{n}$ such that
\begin{eqnarray}
\label{5.5}
(\nabla_{X}R)(Y,Z)W&=&A(X)R(Y,Z)W+B(Y)R(X,Z)W+C(Z)R(Y,X)W\nonumber\\&&+D(W)R(Y,Z)X+g(R(Y,Z)W,X)\sigma,
\end{eqnarray}
where $X$, $Y$, $Z$, $W$ are arbitrary vector fields on $M_{n}$ and the $1-$forms $A$, $B$, $C$, $D$ and the vector field $\sigma$ are defined previously.
	{\defi A non-flat $n-$dimensional differentiable manifold $M_{n}$, $(n>3)$, is called weakly Ricci-symmetric with respect to the semi-symmetric non-metric connection $\nabla$ if there are $1-$forms $\alpha$, $\beta$ and $\gamma$ on $M_{n}$ such that
\begin{equation}
\label{5.6}
(\nabla_{X}\tilde{S})(Y,Z)=\alpha(X)\tilde{S}(Y,Z)+\beta(Y)\tilde{S}(X,Z)+\gamma(Z)\tilde{S}(X,Y),
\end{equation}
where $X$, $Y$, $Z$ are arbitrary vector fields on $M_{n}$.

Contracting (\ref{5.5})with $Y$, we get
\begin{eqnarray}
\label{5.7}
(\nabla_{X}\tilde{S})(Z,W)&=&A(X)\tilde{S}(Z,W)+B(R(X,Z)W)+C(Z)\tilde{S}(W,X)\nonumber\\&&+D(W)\tilde{S}(X,Z)+p(R(X,W)Z),
\end{eqnarray}
where $p$ is defined as $p(X)=g(X,\sigma)$ for arbitrary vector field $X$. 

	$\ddot{O}$zg$\ddot{u}$r \cite{ozgur1} considered weakly symmetric and weakly Ricci-symmetric Kenmotsu manifolds and proved the following theorems:
\begin{thm} There is no weakly symmetric Kenmotsu manifold $M$, $(n>3)$, unless $A+C+D$ is everywhere zero.
\end{thm}
\begin{thm}
 There is no weakly Ricci-symmetric Kenmotsu manifold $M$, $(n>3)$, unless $\alpha+\beta+\gamma$ is everywhere zero.
\end{thm}

	Sular \cite{sular} considered weakly symmetric and weakly Ricci-symmetric Kenmotsu manifolds with respect to the semi-symmetric metric connection and proved the following results:
\begin{thm}
There is no weakly symmetric Kenmotsu manifold $M$ admitting a semi-symmetric metric connection, $(n>3)$, unless $A+C+D$ is everywhere zero.
\end{thm}
\begin{thm}
 There is no weakly Ricci-symmetric Kenmotsu manifold $M$ admitting a semi-symmetric metric connection, $(n>3)$, unless $\alpha+\beta+\gamma$ is everywhere zero.
\end{thm}

Now we consider the weakly symmetric and weakly Ricci-symmetric Kenmotsu manifolds admitting the semi-symmetric non-metric connection $\nabla$ and prove the following theorems:
\begin{thm}
Let $M_{n}$, $(n>3)$ be an $n-$dimensional weakly symmetric Kenmotsu manifold admitting a semi-symmetric non-metric connection $\nabla$ then there is no $M_{n}$, unless $A+C+D$ is everywhere zero.
\end{thm}
\begin{proof}
Replacing $W$ by $\xi$ in (\ref{5.7}) and using (\ref{2.2}), (\ref{2.4}), (\ref{2.9}), (\ref{2.10}), (\ref{2.11}), (\ref{3.5}), (\ref{3.8}) and (\ref{3.9}) we obtain
\begin{eqnarray}
\label{5.8}
(\nabla_{X}\tilde{S})(Z,\xi)&=&-(n-1)A(X)\eta(Z)+\eta(X)B(Z)-\eta(Z)B(X)\nonumber\\&&
-(n-1)C(Z)\eta(X)+D(\xi)S(X,Z)-\eta(Z)p(X)\nonumber\\&&
+p(\xi)g(X,Z)+(n-1)D(\xi)g(\phi{X},Z)-p(\xi)g(\phi{X},Z).
\end{eqnarray}
From (\ref{4.13}) and (\ref{5.8}), it follows that
\begin{eqnarray}
\label{5.9}
&&-(n-1)g(X,Z)+2(n-1)g(\phi{X},Z)-S(X,Z)\nonumber\\&&
=-(n-1)A(X)\eta(Z)+\eta(X)B(Z)-\eta(Z)B(X)\nonumber\\&&
-(n-1)C(Z)\eta(X)+D(\xi)S(X,Z)-\eta(Z)p(X)\nonumber\\&&
+p(\xi)g(X,Z)+(n-1)D(\xi)g(\phi{X},Z)-p(\xi)g(\phi{X},Z).
\end{eqnarray}
Replacing $X$ and $Z$ by $\xi$ in (\ref{5.9}) and using (\ref{2.2}), (\ref{2.4}) and (\ref{2.11}), we get
\begin{equation}
\label{5.10}
A(\xi)+C(\xi)+D(\xi)=0.
\end{equation}
Putting $Z=\xi$ in (\ref{5.7}) and using (\ref{2.2}), (\ref{2.4}), (\ref{2.9}), (\ref{2.10}), (\ref{2.11}), (\ref{3.5}), (\ref{3.8}), we obtain
\begin{eqnarray}
\label{5.11}
(\nabla_{X}\tilde{S})(\xi,W)&=&-(n-1)A(X)\eta(W)+g(X,W)B(\xi)-\eta(W)B(X)+\eta(X)p(W)\nonumber\\&&
-g(\phi{X},W)B(\xi)+C(\xi)S(W,X)+(n-1)C(\xi)g(\phi{W},X)\nonumber\\&&
-(n-1)D(W)\eta(X)-\eta(W)p(X)+2g(\phi{X},W)p(\xi).
\end{eqnarray}
In consequence of (\ref{4.13}) and (\ref{5.11}), we have
\begin{eqnarray}
\label{5.12}
&&-(n-1)g(X,W)+2(n-1)g(\phi{X},W)-S(X,W)\nonumber\\&&
=-(n-1)A(X)\eta(W)+g(X,W)B(\xi)-\eta(W)B(X)+\eta(X)p(W)\nonumber\\&&
-g(\phi{X},W)B(\xi)+C(\xi)S(W,X)+(n-1)C(\xi)g(\phi{W},X)\nonumber\\&&
-(n-1)D(W)\eta(X)-\eta(W)p(X)+2g(\phi{X},W)p(\xi).
\end{eqnarray}
Putting $W=\xi$ in (\ref{5.12}) and using (\ref{2.2}), (\ref{2.4}) and (\ref{2.11}), we get
\begin{eqnarray}
\label{5.13}
&&\eta(X)B(\xi)-p(X)-B(X)-(n-1)C(\xi)\eta(X)\nonumber\\&&
-(n-1)D(\xi)\eta(X)+\eta(X)p(\xi)-(n-1)A(X)=0.
\end{eqnarray}
Replacing $X$ with $\xi$ in (\ref{5.12}) and then using (\ref{2.2}), (\ref{2.4}) and (\ref{2.11}), we find
\begin{equation}
\label{5.14}
-(n-1)A(\xi)\eta(W)-(n-1)C(\xi)\eta(W)-(n-1)D(W)+p(W)-\eta(W)p(\xi)=0.
\end{equation}
Replacing $W$ by $X$ in (\ref{5.14}), we get
\begin{equation}
\label{5.15}
-(n-1)A(\xi)\eta(X)-(n-1)C(\xi)\eta(X)-(n-1)D(X)+p(X)-\eta(X)p(\xi)=0.
\end{equation}
Adding (\ref{5.13}) and (\ref{5.15}) and using (\ref{5.10}), we obtain
\begin{equation}
\label{5.16}
\eta(X)B(\xi)-(n-1)A(X)-(n-1)D(X)-B(X)-(n-1)C(\xi)\eta(X)=0.
\end{equation}
Taking $X=\xi$ in (\ref{5.9}) and then using (\ref{2.2}), (\ref{2.4}) and (\ref{2.11}), we find
\begin{equation}
\label{5.17}
B(Z)-(n-1)A(\xi)\eta(Z)-\eta(Z)B(\xi)-(n-1)C(Z)-(n-1)D(\xi)\eta(Z)=0.
\end{equation}
Replacing $Z$ by $X$ in (\ref{5.17}), we get
\begin{equation}
\label{5.18}
B(X)-(n-1)A(\xi)\eta(X)-\eta(X)B(\xi)-(n-1)C(X)-(n-1)D(\xi)\eta(X)=0.
\end{equation}
Adding (\ref{5.16}) and (\ref{5.18}) and using (\ref{5.10}), we have
\begin{equation}
A(X)+C(X)+D(X)=0.
\end{equation}
Hence the statement of the theorem.
\end{proof}
\begin{thm}
Let $M_{n}$, $(n>3)$, be an $n-$dimensional weakly Ricci-symmetric Kenmotsu manifold admitting a semi-symmetric non-metric connection $\nabla$ then there is no $M_{n}$, unless $\alpha+\beta+\gamma$ is everywhere zero.
\end{thm}
\begin{proof}
Putting $Z=\xi$ in (\ref{5.6}) and then using (\ref{2.11}), (\ref{3.6}) and (\ref{3.9}), we find
\begin{eqnarray}
\label{5.19}
(\nabla_{X}\tilde{S})(Y,\xi)&=&-(n-1)\alpha(X)\eta(Y)-(n-1)\beta(Y)\eta(X)\nonumber\\&&
+\gamma(\xi)S(X,Y)+(n-1)\gamma(\xi)g(\phi{X},Y).
\end{eqnarray}
In consequence of (\ref{4.13}) and (\ref{5.19}), we have
\begin{eqnarray}
\label{5.20}
&&-(n-1)g(X,Y)+2(n-1)g(\phi{X},Y)-S(X,Y)\nonumber\\&&
=-(n-1)\alpha(X)\eta(Y)-(n-1)\beta(Y)\eta(X)\nonumber\\&&
+\gamma(\xi)S(X,Y)+(n-1)\gamma(\xi)g(\phi{X},Y).
\end{eqnarray}
Taking $X=Y=\xi$ in (\ref{5.20}) and using (\ref{2.2}), (\ref{2.4}) and (\ref{2.11}), we find
\begin{equation}
\label{5.21}
\alpha(\xi)+\beta(\xi)+\gamma(\xi)=0.
\end{equation}
Replacing $X$ by $\xi$ in (\ref{5.20}) and using (\ref{2.2}), (\ref{2.4}), (\ref{2.11}) and (\ref{5.21}), we get
\begin{equation}
\label{5.22}
\beta(Y)=\beta(\xi)\eta(Y).
\end{equation}
Again replacing $Y$ by $\xi$ in (\ref{5.20}) and using (\ref{2.2}), (\ref{2.4}), (\ref{2.11}) and (\ref{5.21}), we obtain
\begin{equation}
\label{5.23}
\alpha(X)=\alpha(\xi)\eta(X).
\end{equation}
From (\ref{2.2}), (\ref{2.8}), (\ref{2.11}), (\ref{2.12}), (\ref{3.3}) and (\ref{3.9}), it follows that
\begin{equation}
(\nabla_{\xi}\tilde{S})(\xi,X)=0.
\end{equation}
In view of (\ref{5.6}), above equation becomes
\begin{equation}
\label{5.24}
\alpha(\xi)\tilde{S}(\xi,X)+\beta(\xi)\tilde{S}(\xi,X)+\gamma(X)\tilde{S}(\xi,\xi)=0.
\end{equation}
With the help of (\ref{2.2}), (\ref{2.4}), (\ref{2.11}) and (\ref{3.9}), equation (\ref{5.24}) gives
\begin{equation}
\label{5.25}
\gamma(X)=\gamma(\xi)\eta(X).
\end{equation}
Adding (\ref{5.22}), (\ref{5.23}) and (\ref{5.25}) and using (\ref{5.21}), we get the statement of the theorem.
\end{proof}


\begin{thebibliography}{99}
\bibitem{friedmann}
A. Friedmann and J. A. Schouten, $\ddot U$ber die Geometric der  halbsymmetrischen  $\ddot U$bertragung, Math., Zeitschr., 21 (1924), 211-223.
\bibitem{hayden}
H. A. Hayden, Subspace of a space with torsion, Proceedings of London Mathematical  Society II Series 34 (1932), 27-50.
\bibitem{yano}
K. Yano, On semi-symmetric metric connections, Rev. Roumaine Math. Pures Appl., 15 (1970), 1579-1586.
\bibitem{imai1}
T. Imai, Notes on semi-symmetric metric connections, Tensor N. S., 24 (1972), 293-296.
\bibitem{blair}
D. E. Blair, Contact manifold in Riemannian Geometry, Lecture notes in Mathematics, 509, Springer Verlag Berlin,1976.
\bibitem{imai2}
T. Imai, Hypersurfaces of a Riemannian manifold with semi-symmetric metric connection, Tensor N. S., 23 (1972), 300-306.  
\bibitem{srivastava}
B. B. Sinha and A. K. Srivastava, Curvature on Kenmotsu manifold, Indian J. Pure and Appl. Math. 22,(1) (1991),23-28.
\bibitem{ucd}
U. C. De, On $\phi$ symmetric Kenmotsu manifolds, International electronic Jour. of geometry.1 (2008),33-38.
\bibitem{pathak}
U. C. De and G. Pathak, On three dimensional Kenmotsu manifolds, Indian Jour. Pure Appl. Math ,35(2004) 159-165.
\bibitem{chaubeyadd2}
S. K. Chaubey and R. H. Ojha : On the m-projective curvature tensor of a Kenmotsu manifold, Differential Geometry-Dynamical Systems, 12, 2010, 52-60.
\bibitem{addsk3}
S. K. Chaubey and C. S. Prasad, ON generalized $\phi$ recurrent Kenmotsu manifolds, TWMS J. App. Eng. Math. V.5, N.1, 2015, pp. 1-9.
\bibitem{skchaubey_2012}
S. K. Chaubey, S. Prakash and R. Nivas, Some properties of $m-$projective curvature tensor in Kenmotsu manifolds, Bulletin of Mathematical Analysis and Applications, 4 (3), (2012), 48-56.
\bibitem{agashe1}
N. S. Agashe and M. R. Chafle, A semi-symmetric non-metric connection in a Riemannian manifold, Indian J. Pure Appl. Math., 23 (1992), 399--409.
\bibitem{agashe2}
N. S. Agashe and M. R. Chafle, On submanifolds of a Riemannian manifold with semi-symmetric non-metric connection, Tensor N. S., 55, no. 2 (1994), pp. 120--130.
\bibitem{pandey}
L. K. Pandey, and R. H. Ojha, Semi-symmetric metric and non-metric connections in Lorentzian Paracontact manifold, Bull. Cal. Math. Soc., 93, no. 6 (2001), pp. 497--504.
\bibitem{de2}
U. C. De and J. Sengupta, On a type of semi-symmetric non-metric connection, Bull. Cal. Math. Soc., 92, no. 5 (2000), pp. 375--384.
\bibitem{de1}
U. C. De and D. Kamilya, Hypersurfaces of a Riemannian manifold with semi-symmetric non-metric connection, J. Indian Inst. Sci., 75, (1995), pp. 707--710.
\bibitem{kenmotsu}
K. Kenmotsu, A class of almost contact Riemannian  manifolds , Tohoku Math. J., 24 (1972), 93-103.
\bibitem{chaturvedi}
B. B. Chaturvedi and P. N. Pandey, Semi-symmetric non-metric connection on a K$\ddot{a}$hler manifold, Differential Geometry- Dynamical Systems, 10, (2008), pp. 86--90.
\bibitem{chaubey1}
S. K. Chaubey and R. H. Ojha, On semi-symmetric non-metric and quarter-symmetric metric connections, Tensor N. S., 70, 2.
(2008) 202--213.
\bibitem{chaubeyadd1}
S. K. Chaubey, On a semi-symmetric non-metric connection, https://arxiv.org/abs/1711.01035.
\bibitem{sengupta}
J. Sengupta, U. C. De and T. Q. Binh, On a type of semi-symmetric non-metric connection on a Riemannian manifold, Indian J. Pure Appl. Math., 31, no. (1-2) (2000), pp. 1659--1670.
\bibitem{dubey}
A. K. Dubey, S. K. Chaubey and R. H. Ojha, On semi-symmetric non-metric connection, International Mathematical Forum, 5, no. 15 (2010), pp. 731--737.
\bibitem{kumar}
A. Kumar and S. K. Chaubey, A semi-symmetric non-metric connection in a generalised cosymplectic manifold, Int. Journal of Math. Analysis, 4, no. 17 (2010), pp. 809--817.
\bibitem{jaiswal}
J. P. Jaiswal and R. H. Ojha, Some properties of K- contact Riemannian manifolds admitting a semi-symmetric non-metric connection, Filomat 24:4, (2010), pp. 9--16.
\bibitem{chaubey3}
S. K. Chaubey and R. H. Ojha, On a semi-symmetric non-metric connection, Filomat, 26:2, (2012), pp. 63-69.
\bibitem{skchaubey_2011}
S. K. Chaubey, Almost contact metric manifolds admitting semi-symmetric non-metric connection, Bulletin of Mathematical Analysis and Applications, 3 (2), (2011), 252-260.
\bibitem{addsk4}
S. K. Chaubey and A. C. Pandey, Some properties of a semi-symmetric non-metric connection on a Sasakian manifold, Int. J. Contemp. Math. Sciences, Vol. 8, 2013, no. 16, 789 - 799.
\bibitem{de3}
U. C. De and N. Guha, On generalized recurrent manifolds, Proc. Math. Soc., 7 (1991), 7-11.
\bibitem{chaki}
M. C. Chaki, On pseudo symmetric manifolds. Analele Stiintifice ale Universitatii Al. I.
Cuza din Iasi, 33: 53–58, 1987.
\bibitem{ozgur1}
C. $\ddot{O}$zg$\ddot{u}$r, On weakly symmetric Kenmotsu manifolds, Differential Geometry- Dynamical Systems, 8 (2006), 204-209.
\bibitem{ozgur2}
C. $\ddot{O}$zg$\ddot{u}$r, On generalized recurrent Kenmotsu manifolds, World Applied Sciences J., 2 (2007), no. 1, 29-33.
\bibitem{tamassy1}
L. Tam$\acute{a}$ssy and T. Q. Binh, On weakly symmetric and weakly projective symmetric Riemannian manifolds, Colloq. Math. Soc. J. Bolyai, 56 (1992), 663-670.
\bibitem{tamassy2}
L. Tam$\acute{a}$ssy and T. Q. Binh, On weak symmetries of Einstein and Sasakian manifolds, Tensor N. S., 53 (1993), no. 1, 140-148.
\bibitem{sular}
S. Sular, Some properties of a Kenmotsu manifold with a semi-symmetric metric connection, International Electronic Journal of Geometry, 3 (2010), no. 1, 24-34.
\bibitem{tripathi2}
M. M. Tripathi, A new connection in a Riemannian manifold, International Electronic Journal of Geometry, 1 No. 1 (2008), pp. 15-24.
\bibitem{ozgur3}
C. $\ddot{O}$zg$\ddot{u}$r, On submanifolds of a Riemannian manifold with a semi-symmetric non-metric connection, Kuwait J. Sci. Engg., (Accepted).

\bibitem{add1}
A. K. Dubey, R. H. Ojha and S. K.  Chaubey, Some properties of quarter-symmetric non-metric connection in a K$\ddot{a}$hler manifold, Int. J. Contemp. Math. Sciences, Vol. 5, 2010, no. 20, 1001 - 1007.

\bibitem{add2}
S. K. Chaubey and R. H. Ojha, On quater-symmetric non-metric connection on almost Hermitian manifold, Bulletin of Mathematical Analysis and Applications, Volume 2 (2) (2010), 77-83.

\bibitem{add3}
S. K. Chaubey, On weakly $m−$ projectively symmetric manifolds, Novi Sad J. Math, 42 (1) (2012), 67-79.

\bibitem{add4}
S. K. Chaubey and Ashok Kumar, Semi-symmetric metric $T-$connection in an almost contact metric manifold, International Mathematical Forum, 5(23) (2010), 1121-1129.

\bibitem{add5}
S. K. Yadav, Pankaj and  S. K. Chaubey, Riemannian manifolds admitting a projective semi-symmetric connection, https://arxiv.org/pdf/1710.00622.pdf.

\bibitem{add6}
Pankaj, S. K. Chaubey and R. Prasad, Trans-Sasakian manifolds with respect to a non-symmetric non-metric connection, Global Journal of Advanced Research on Classical and Modern Geometries, Vol.7, (2018), Issue 1, pp.1-10.



\end{thebibliography}
\end{document}